\documentclass{aic}

\usepackage{amsthm,amssymb,amsmath,graphicx,mathtools,enumerate}
\usepackage[shortlabels]{enumitem}
\usepackage[british]{babel}
\usepackage[capitalise]{cleveref}

\theoremstyle{plain}
\newtheorem{theorem}{Theorem}[section]

\newtheorem{lemma}[theorem]{Lemma}

\newtheorem{claim}[theorem]{Claim}

\theoremstyle{definition}
\newtheorem{definition}[theorem]{Definition}

\newcommand{\eps}{\varepsilon}
\newcommand{\sm}{\setminus}
\newcommand{\N}{\mathbb{N}}
\newcommand{\RdG}{{G'}}

\newcommand{\ue}{u_\text{e}}
\newcommand{\uo}{u_\text{o}}

\newcommand{\ceiling}[1]{\left\lceil#1\right\rceil}

\newcommand{\rioflat}{{(12+\sqrt{8})}/{17}}

\renewcommand{\subset}{\subseteq}

\aicAUTHORdetails{%
  title = {Upper Density of Monochromatic Infinite Paths}, 
  author = {Jan Corsten, Louis DeBiasio, Ander Lamaison, and Richard Lang},
  plaintextauthor = {Jan Corsten, Louis DeBiasio, Ander Lamaison, Richard Lang},
  keywords = {infinite graph, Ramsey, upper density, regularity lemma.},
}   

\aicEDITORdetails{%
   year={2019},
   number={4},
   received={10 August 2018},   
   revised={24 December 2018},    
   published={30 October 2019},  
   doi={10.19086/aic.10810},      
}   

\begin{document}

\begin{frontmatter}[classification=text]


\author[jc]{Jan Corsten\thanks{Research supported by an LSE PhD scholarship.}}
\author[ld]{Louis DeBiasio\thanks{Research supported in part by Simons Foundation Collaboration Grant \# 283194.}}
\author[al]{Ander Lamaison\thanks{Research supported by the Deutsche Forschungsgemeinschaft (DFG, German Research Foundation) under Germany's Excellence Strategy - The Berlin Mathematics Research Center MATH+ (EXC-2046/1, project ID: 390685689).}}
\author[rl]{Richard Lang\thanks{Research supported by EPSRC, grant no. EP/P002420/1.}}

\begin{abstract}
We prove that in every $2$-colouring of the edges of $K_\N$ there exists a monochromatic infinite path $P$ such that $V(P)$ has upper density at least $\rioflat \approx 0.87226$ and further show that this is best possible.
This settles a problem of Erd\H{o}s and Galvin.
\end{abstract}
\end{frontmatter}

\section{Introduction}
Given a complete graph $K_n$, whose edges are coloured in red and blue, what is the longest monochromatic path one can find?
Gerencs\'er and Gy\'arf\'as~\cite{GG} proved that there is always a monochromatic path on $\ceiling{(2n+1)/3}$ vertices, which is best possible.
It is natural to consider a density analogue of this result for 2-colourings of $K_{\N}$.
The \emph{upper density} of a graph $G$ with $V(G) \subset \N$ is defined as $$\bar{d}(G) = \limsup_{t\rightarrow\infty} \frac{|V(G) \cap \{1,2,\ldots,t\}|}{t}.$$
The \emph{lower density} is defined similarly in terms of the infimum and we speak of the \emph{density}, whenever lower and upper density coincide.

Erd\H{o}s and Galvin~\cite{EG} described a $2$-colouring of $K_{\N}$ in which every monochromatic infinite path has lower density $0$ and thus we restrict our attention to upper densities.
Rado~\cite{R} proved that every $r$-edge-coloured $K_{\N}$ contains $r$ vertex-disjoint monochromatic paths which together cover all vertices. In particular, one of them must have upper density at least $1/r$.
Erd\H{o}s and Galvin~\cite{EG} proved that in every $2$-colouring of $K_{\N}$ there is a monochromatic path $P$ with $\bar{d}(P)\geq 2/3$. 
Moreover, they constructed a $2$-colouring of $K_{\N}$ in which every monochromatic path $P$ has upper density at most $8/9$.
DeBiasio and McKenney~\cite{DM16} recently improved the lower bound to $3/4$ and conjectured the correct value to be $8/9$.
Progress towards this conjecture was made by Lo, Sanhueza-Matamala and Wang~\cite{LSW18}, who raised the lower bound to $(9+\sqrt{17})/16\approx 0.82019$.

We prove that the correct value is in fact $\rioflat\approx 0.87226$. 

\begin{theorem}\label{thm:upper-bound}
	There exists a 2-colouring of the edges of $K_\N$ such that every monochromatic path has upper density at most $\rioflat$.
\end{theorem}

\begin{theorem}\label{thm:lower-bound}
	In every $2$-colouring of the edges of $K_\N$, there exists a monochromatic path of upper density at least $\rioflat$.
\end{theorem}

Now that we have solved the problem for two colours, it would be very interesting to make any improvement on Rado's lower bound of $1/r$ for $r\geq 3$ colours (see~\cite[Corollary 3.5]{DM16} for the best known upper bound).  In particular for three colours, the correct value is between $1/3$ and $1/2$.

\section{Notation}
We write $\N$ to be the positive integers with the standard ordering.
Throughout the paper when referring to a finite graph on $n$ vertices, it is always assumed that the vertex set is $[n] = \{1, 2, \dots, n\}$ and that it is ordered in the natural way.
An \emph{infinite path} $P$ is a graph with vertex set $V(P)=\{v_i: i\in \N\}$ and edge set $E(P)=\{v_iv_{i+1}: i\in \N\}$.  While paths are defined to be one-way infinite, all of the results mentioned above on upper density of monochromatic infinite paths apply equally well to two-way infinite paths.
For a graph $G$ with $V(G) \subset \N$ and $t\in \N$, we define $$d(G,t) = \frac{|V(G) \cap [t]|}{t}.$$
Thus we can express the {upper density} of $G$ as
$\bar{d}(G) = \limsup_{t\rightarrow\infty} d(G,t).$

\section{Upper bound}\label{sec:upper-bound}
In this section, we will prove Theorem~\ref{thm:upper-bound}. Let $q>1$ be a real number, whose exact value will be chosen later on.
We start by defining a colouring of the edges of the infinite complete graph.
Let $A_0,A_1,\dots$ be a partition of $\N$, such that every element of $A_i$ precedes every element of $A_{i+1}$ and $|A_i| = \lfloor q^i \rfloor$.
We colour the edges of $G=K_\N$ such that every edge $uv$ with $u\in A_i$ and $v\in A_j$ is red if $\min\{i,j\}$ is odd, and blue if it is even.
A straightforward calculation shows that for $q=2$, every monochromatic path $P$ in $G$ satisfies $\bar{d}(P)\leq 8/9$ (see Theorem 1.5 in \cite{EG}).	
We will improve this bound by reordering the vertices of $G$ and then optimizing the value of $q$.

For convenience, we will say that the vertex $v \in A_i$ is red if $i$ is odd and blue if $i$ is even. 
We also denote by $B$ the set of blue vertices and by $R$ be the set of red vertices. Let $b_i$ and $r_i$ denote the $i$-th blue vertex and the $i$-th red vertex, respectively.
We define a monochromatic red matching $M_r$ by forming a matching between $A_{2i-1}$ and the first $|A_{2i-1}|$ vertices of $A_{2i}$ for each $i\geq 1$. Similarly, we define a monochromatic blue matching $M_b$ by forming a matching between $A_{2i}$ and the first $|A_{2i}|$ vertices of $A_{2i+1}$ for each $i\geq 0$.

\begin{figure}[ht]
	\centering
	\includegraphics[scale=.75]{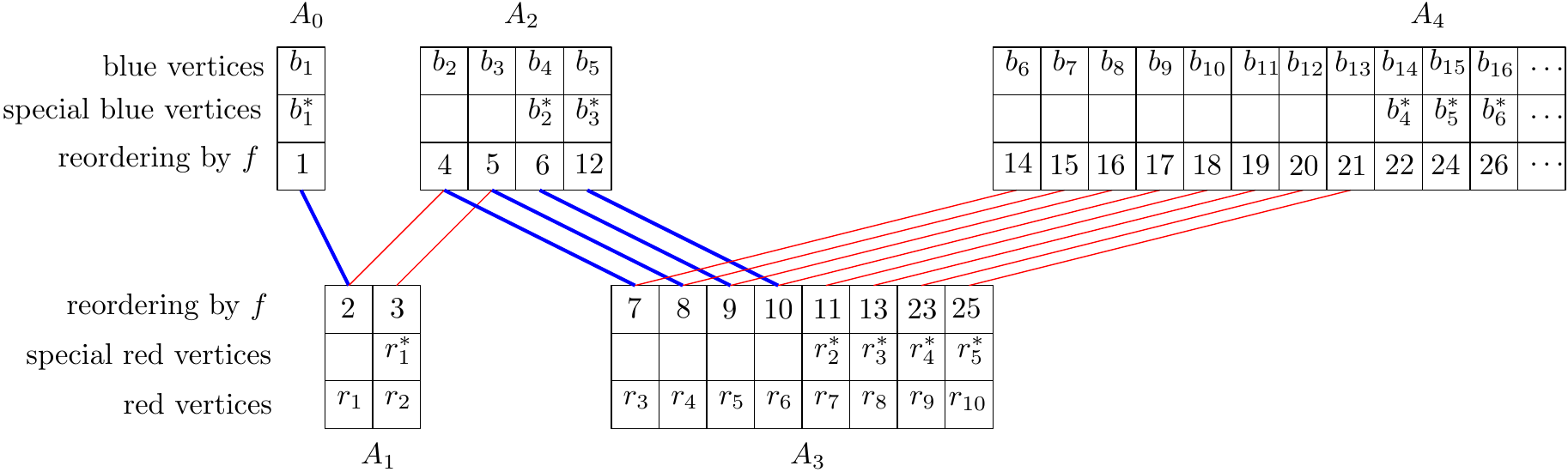}
	\caption{The colouring for $q=2$ and the reordering by $f$.}\label{fig:geo}
\end{figure}
Next, let us define a bijection $f\colon \N \to V(G)$, which will serve as a reordering of $G$.
Let $r_t^*$ denote the $t$-th red vertex not in $M_b$, and $b_t^*$ denote the $t$-th blue vertex not in $M_r$. The function $f$ is defined as follows.
We start enumerating blue vertices, in their order, until we reach $b_1^*$. 
Then we enumerate red vertices, in their order, until we reach $r_1^*$. 
Then we enumerate blue vertices again until we reach $b_2^*$. 
We continue enumerating vertices in this way, changing colours whenever we find an $r_t^*$ or a $b_t^*$.
(See Figure~\ref{fig:geo}.)
Finally, for every $H \subset G$, we define
\[\bar{d}(H;f) = \limsup_{t\rightarrow\infty} \frac{|V(H) \cap f([t])|}{t}.\]
Note that $\bar d(H;f)$ is the upper density of $H$ in the reordered graph $f^{-1}(G)$. 

\begin{claim}\label{cla:faithful}
	Let $P_r$ and $P_b$ be infinite monochromatic red and blue paths in $G$, respectively. Then $\bar{d}(P_r;f)\leq \bar{d}(M_r;f)$ and $\bar{d}(P_b;f)\leq \bar{d}(M_b;f)$.
\end{claim}

\begin{claim}\label{cla:upper-bound}
	We have $$\bar{d}(M_r;f),~\bar{d}(M_b;f)  \leq \frac{q^2+2q-1}{q^2+3q-2}.$$
\end{claim}
We can easily derive \cref{thm:upper-bound} from these two claims.
Note that the rational function in Claim~\ref{cla:upper-bound} evaluates to $\rioflat$ at $q \coloneqq \sqrt{2}+1$.
It then follows from Claim~\ref{cla:faithful} and~\ref{cla:upper-bound}, that every monochromatic path $P$ in $G$ satisfies $ \bar{d}(P;f) \leq \rioflat$.
Thus we can define the desired colouring of $K_{\N}$, by colouring each edge $ij$ with the colour of the edge $f(i)f(j)$ in $G$.

It remains to prove Claim~\ref{cla:faithful} and~\ref{cla:upper-bound}.
The intuition behind Claim~\ref{cla:faithful} is that in every monochromatic red path $P_r$ there is a red matching with the same vertex set, and that $M_r$ has the largest upper density among all red matchings, as it contains every red vertex and has the largest possible upper density of blue vertices.
Note that the proof of Claim~\ref{cla:faithful} only uses the property that $f$ preserves the order of the vertices inside $R$ and inside $B$.
\begin{proof}[Proof of Claim~\ref{cla:faithful}]
	We will show $\bar{d}(P_r;f)\leq \bar{d}(M_r;f)$.
	(The other case is analogous.)
	We prove that, for every positive integer $k$, we have $|V(P_r)\cap f([k])|\leq |V(M_r)\cap f([k])|$.
	Assume, for contradiction, that this is not the case and let $k$ be the minimum positive integer for which the inequality does not hold.
	Every red vertex is saturated by $M_r$, so $|V(P_r)\cap f([k])\cap B|>|V(M_r)\cap f([k])\cap B|$. By the minimality of $k$, $f(k)$ must be in $P_r$ but not in $M_r$, and in particular it must be blue. 
	
	Let $f(k)\in A_{2i}$. Since $f(k) \not \in M_r$, we know that $f(k)$ is not among the first $|A_{2i-1}|$ vertices of $A_{2i}$. Therefore, since $f$ preserves the order of the vertices inside $B$, $f([k])$ contains the first $|A_{2i-1}|$ blue vertices in $A_{2i}$, and hence
	\begin{equation}\label{eq:1}
	|V(P_r)\cap f([k])\cap B|>|V(M_r)\cap f([k])\cap B|=\sum_{j=1}^i|A_{2j-1}|. 
	\end{equation}
	On the other hand, every edge between two blue vertices is blue, so the successor of every blue vertex in $P_r$ is red, and in particular there is a red matching between $V(P_r)\cap B$ and $R$ saturating $V(P_r)\cap B$. So by \eqref{eq:1}, the number of red neighbours of $V(P_r)\cap f([k])\cap B$ is at least $|V(P_r)\cap f([k])\cap B|>\sum_{j=1}^i|A_{2j-1}|$.
	Observe that by the definition of $f$, we have $V(P_r)\cap f([k])\cap B\subseteq \bigcup_{j=0}^iA_{2j}$. 
	Hence the red neighbourhood of $V(P_r)\cap f([k])\cap B$ is contained in $\bigcup_{j=1}^iA_{2j-1}$, 
	a contradiction.
\end{proof}
\begin{proof}[Proof of Claim~\ref{cla:upper-bound}]
	Let $\ell_r(t)$ and $\ell_b(t)$ denote the position of $r_t^*$ among the red vertices and of $b_t^*$ among the blue vertices, respectively.
	In other words, let $\ell_r(t)=i$ where $r_t^*=r_i$ and $\ell_b(t)=j$ where $b_t^*=b_j$ (so for example in Figure \ref{fig:geo}, $\ell_r(4)=9$ and $\ell_b(4)=14$). Note that  $f(\ell_b(t)+\ell_r(t))=r_t^*$, so for $\ell_b(t-1)+\ell_r(t-1)\leq k \leq \ell_b(t)+\ell_r(t)-1$, $f([k])$ has exactly $t-1$ vertices outside of $M_b$ and at least $t-1$ vertices outside of $M_r$.  As a consequence, we obtain
	\begin{equation}\label{ub}
	\bar{d}(M_r;f),~\bar{d}(M_b;f)\leq \limsup_{k \to \infty} (1 - h(k))= \limsup\limits_{t\rightarrow\infty} \left(1-\frac{t-1}{\ell_r(t)+\ell_b(t)-1}\right),
	\end{equation}
    where $h(k) = (t-1)/k$ if $ \ell_b(t-1) + \ell_r(t-1) \leq k \leq \ell_b(t) + \ell_r(t)-1$. 
	It is easy to see that
	\begin{align*}
	\ell_r(t) = t+\sum\limits_{j=0}^i|A_{2j}|\quad &\text{for}\quad\sum\limits_{j=0}^{i-1}(|A_{2j+1}|-|A_{2j}|)<t\leq \sum\limits_{j=0}^i(|A_{2j+1}|-|A_{2j}|), \text{ and}\\
	\ell_b(t) = t+\sum\limits_{j=1}^i|A_{2j-1}|\quad &\text{for}\quad\sum\limits_{j=1}^{i-1}(|A_{2j}|-|A_{2j-1}|)<t-|A_0|\leq\sum\limits_{j=1}^i(|A_{2j}|-|A_{2j-1}|).
	\end{align*}
    Note that $\ell_r(t)-t$ and $\ell_b(t)-t$ are piecewise constant and non-decreasing. We claim that, in order to compute the right hand side of \eqref{ub}, it suffices to consider values of $t$ for which $\ell_r(t)-t>\ell_r(t-1)-(t-1)$ or $\ell_b(t)-t>\ell_b(t-1)-(t-1)$. This is because we can write \[1-\frac{t-1}{\ell_r(t)+\ell_b(t)-1}=\frac12+\frac{(\ell_r(t)-t)+(\ell_b(t)-t)+1}{2(\ell_r(t)+\ell_b(t)-1)}.\] In this expression, the second fraction has a positive, piecewise constant numerator and a positive increasing denominator. Therefore, the local maxima are attained precisely at the values for which the numerator increases.
    We will do the calculations for the case when $\ell_r(t)-t>\ell_r(t-1)-(t-1)$ (the other case is similar), in which we have
	\begin{align*}
	t &= 1+\sum\limits_{j=0}^{i-1}(|A_{2j+1}|-|A_{2j}|)=1+\sum\limits_{j=0}^{i-1}(1+o(1))q^{2j}(q-1)=(1+o(1))\frac{q^{2i}}{q+1},\\
	\ell_r(t) &= t+\sum\limits_{j=0}^i|A_{2j}|=(1+o(1))\left(\frac{q^{2i}}{q+1}+\sum\limits_{j=0}^iq^{2j}\right)=(1+o(1))\frac{(q^2+q-1)q^{2i}}{q^2-1}, \text{ and}\\
	\ell_b(t) &= t+\sum\limits_{j=1}^i|A_{2j-1}|=(1+o(1))\left(\frac{q^{2i}}{q+1}+\sum\limits_{j=1}^iq^{2j-1}\right)=(1+o(1))\frac{(2q-1)q^{2i}}{q^2-1}. 
	\end{align*}
	Plugging this into \eqref{ub} gives the desired result.
\end{proof}
    
\section{Lower bound}\label{sec:lower-bound}
This section is dedicated to the proof of \cref{thm:lower-bound}. 
A \emph{total colouring} of a graph $G$ is a colouring of the vertices and edges of $G$. Due to an argument of Erd\H{o}s and Galvin, the problem of bounding the upper density of monochromatic paths in edge coloured graphs can be reduced to the problem of bounding the upper density of monochromatic path forests in totally coloured graphs.

\begin{definition}[Monochromatic path forest]
	Given a totally coloured graph $G$, a forest $F \subset G$ is said to be a \emph{monochromatic path forest} if $\Delta(F) \leq 2$ and there is a colour $c$ such that all leaves, isolated vertices, and edges of $F$ receive colour $c$.
\end{definition}

\begin{lemma}\label{lem:path-forest}
	For every $\gamma>0$ and $k \in \N$, there is some $n_0 = n_0(k,\gamma)$ so that the following is true for every $n \geq n_0$. For every total $2$-colouring of $K_n$, there is an integer $ t \in [k,n]$ and a monochromatic path forest $F$ with $d(F,t) \geq \rioflat - \gamma$. 
\end{lemma}

Some standard machinery related to Szemer\'edi's regularity lemma, adapted to the ordered setting, will allow us to reduce the problem of bounding the upper density of monochromatic path forests to the problem of bounding the upper density of monochromatic simple forests.

\begin{definition}[Monochromatic simple forest]
	Given a totally coloured graph $G$, a forest $F \subset G$ is said to be a \emph{monochromatic simple forest} if $\Delta(F) \leq 1$ and there is a colour $c$ such that all edges and isolated vertices of $F$ receive colour $c$ and at least one endpoint of each edge of $F$ receives colour $c$.
\end{definition}
\begin{lemma}\label{lem:simple-forest}
	For every $\gamma > 0$, there exists $k_0,N \in \N$ and $\alpha > 0$ such that the following holds for every integer $k \geq k_0$.
	Let $G$ be a totally $2$-coloured graph on $kN$ vertices with minimum degree at least $(1-\alpha)kN$.
	Then there exists an integer  $t \in [k/8,kN]$ and a monochromatic simple forest $F$ such that $d(F,t) \geq \rioflat - \gamma$.
\end{lemma}

The heart of the proof is Lemma~\ref{lem:simple-forest}, which we shall prove in Section~\ref{sec:simple-forest}. But first, in the next two sections, we show how to deduce Theorem~\ref{thm:lower-bound} from Lemmas~\ref{lem:path-forest} and~\ref{lem:simple-forest}.

\subsection{From path forests to paths}\label{sec:proof-lower-bound}
In this section we use Lemma~\ref{lem:path-forest} to prove Theorem~\ref{thm:lower-bound}.
Our exposition follows that of Theorem 1.6 in~\cite{DM16}.
\begin{proof}[Proof of Theorem~\ref{thm:lower-bound}]
	Fix a $2$-colouring of the edges of $K_\N$ in red and blue. 
	We define a $2$-colouring of the vertices
	by colouring $n \in \N$ red if there are infinitely many $m \in \N$ such that the edge $nm$ is red and blue otherwise.
	
	\textbf{Case 1.} Suppose there are vertices $x$ and $y$ of the same colour, say red, and a finite set $S \subset \N$ such that there is no red path disjoint from $S$ which connects $x$ to $y$.
	
	We partition $\N \sm S$ into sets $X, Y, Z$, where $x' \in X$ if and only if there is a red path, disjoint from $S$, which connects $x'$ to $x$ and $y' \in Y$ if and only if there is a red path disjoint from $S$ which connects $y$ to $y'$. 
	Note that every edge from $X \cup Y$ to $Z$ is blue. Since $x$ and $y$ are coloured red, both $X$ and $Y$ are infinite, and by choice of $x$ and $y$ all edges in the bipartite graph between $X$ and $Y \cup  Z$ are blue. Hence
	there is a blue path with vertex set $X \cup Y \cup  Z = \N \sm S$.
	
	\textbf{Case 2.} Suppose that for every pair of vertices $x$ and $y$ of the same colour $c$, and every finite set $S \subset \N$, there is a path from $x$ to $y$ of colour $c$ which is disjoint from $S$.
	
	Let $\gamma_n$ be a sequence of positive reals tending to zero, and let $a_n$ and $k_n$ be increasing sequences of integers such that
	\begin{equation*}
	\text{$a_n \geq n_0(k_n,\gamma_n)$ and  $k_n/(a_1+\dots+a_{n-1}+k_n) \rightarrow 1$,}
	\end{equation*}
	where $n_0(k,\gamma)$ is as in Lemma~\ref{lem:path-forest}.
	Let $\N = (A_i)$ be a partition of $\N$ into consecutive intervals with $|A_n| = a_n$.
	By Lemma~\ref{lem:path-forest} there are monochromatic path forests $F_n$ with $V(F_n) \subset A_n$ and initial segments $I_n \subset A_n$ of length at least $k_n$ such that
	$$|V(F_n) \cap I_n| \geq \left(\frac{12+\sqrt{8}}{17}-\gamma_n\right)|I_n|.$$
	It follows that for any $G \subset K_\N$ containing infinitely many $F_n$'s we have
	$$\bar{d}(G) \geq \limsup_{n\rightarrow \infty} \frac{|V(F_n) \cap I_n|}{a_1 + \dots + a_{n-1} + |I_n|} \geq \limsup_{n\rightarrow \infty} \frac{12+\sqrt{8}}{17} - \gamma_n = \frac{12+\sqrt{8}}{17}.$$
	
	By the pigeonhole principle, there are infinitely many $F_n$'s of the same colour, say blue. We will recursively construct a blue path $P$ which contains infinitely many of these $F_n$'s. To see how this is done, suppose we have constructed a finite initial segment $p$ of $P$. We will assume as an inductive hypothesis that $p$ ends at a blue vertex $v$.
	Let $n$ be large enough that $\min (A_n)$ is greater than every vertex in $p$, and $F_n$ is blue. Let $F_n = \{P_1, \dots, P_s\}$ for some $s \in \N$ and let $w_i, w_i'$ be the endpoints of the path $P_i$ (note that $w_i$ and $w_i'$ could be equal) for every $i \in [s]$.
	By the case assumption, there is a blue path $q_1$ connecting $v$ to $w_1$, such that $q_1$ is disjoint from $A_1 \cup \dots \cup  A_n$. Similarly, there is a blue path $q_2$ connecting $w_1'$ to $w_2$, such that $q_2$ is disjoint from $A_1 \cup \dots \cup A_n \cup \{q_1\}$. Continuing in this fashion, we find disjoint blue paths $q_3, \dots , q_s$ such that $q_i$ connects $w_{i-1}'$ to $w_i$. Hence, we can extend $p$ to a path $p'$ which contains all of the vertices of $F_n$ and ends at a blue vertex.
\end{proof}

\subsection{From simple forests to path forests}\label{sec:path-forest}
In this section we use Lemma~\ref{lem:simple-forest} to prove Lemma~\ref{lem:path-forest}.  
The proof is based on Szemer\'edi's Regularity Lemma, which we introduce below.
The main difference to standard applications of the Regularity Lemma is, that we have to define an ordering of the reduced graph, which approximately preserves densities.
This is done by choosing a suitable initial partition.

Let $G=(V,E)$ be a graph and~$A$ and~$B$ be non-empty, disjoint subsets of $V$.
We write $e_G(A,B)$ for the number of edges in~$G$ with one vertex in~$A$ and one in~$B$ and define the \emph{density} of the pair $(A,B)$ to be $d_{G}(A,B)=e_G(A,B)/(|A||B|)$. 
The pair $(A,B)$ is \emph{$\eps$-regular} (in~$G$) if we have $|d_{G}(A',B') - d_{G}(A,B)| \leq \eps$ for all $A'\subseteq A$ with $|A'|\ge\eps|A|$ and $B'\subseteq B$ with $|B'|\ge\eps |B|$.
It is well-known (see for instance \cite{Hax}) that dense regular pairs contain almost spanning paths. We include a proof of this fact for completeness.

\begin{lemma}\label{fac:almost-spanning-path}
For $0 < \eps < 1/4$ and $d \geq 2\sqrt{\eps}+\eps$, every $\eps$-regular pair $(A,B)$ with density at least $d$ contains a path with both endpoints in $A$ and covering all but at most $2\sqrt{\eps}|A|$ vertices of $A \cup B$.
\end{lemma}

\begin{proof}
	We will construct a path  $P_k = (a_1b_1\dots a_{k})$ for every $ k =1, \ldots, \lceil (1- \sqrt \eps) |A| \rceil$ such that $B_k \coloneqq  N(a_k) \sm V(P_k)$ has size at least $\eps |B|$.
	As $d \geq \eps$, this is easy for $k = 1$. Assume now that we have constructed $P_k$ for some $ 1 \leq k < (1-\sqrt{\eps})|A|$. We will show how to extend $P_k$ to $P_{k+1}$. 
	By $\eps$-regularity of $(A,B)$, the set $\bigcup_{b \in B_k} N(b)$ has size at least $(1-\eps)|A|$.
	So $A' \coloneqq \bigcup_{b \in B_k} N(b) \sm V(P_k)$ has size at least $(\sqrt{\eps}-\eps) |A| \geq \eps |A|$.
	Let $B' = B \sm V(P_k)$ and note that $|B'|\geq \sqrt{\eps} |B|$ as $k < (1-\sqrt{\eps})|A|$ and $|A|=|B|$.
	By $\eps$-regularity of $(A,B)$, there exists $a_{k+1} \in A'$ with at least $(d-\eps)|B'| \geq 2\eps|B|$ neighbours in $B'$.
	Thus we can define $P_{k+1} = (a_1b_1\dots a_{k}b_{k}a_{k+1})$, where $b_{k} \in B_k \cap N(a_{k+1})$.
\end{proof}

{A family of disjoint subsets }$\{V_i\}_{i \in [m]}$ of a set $V$ is said to \emph{refine} {a partition} $\{W_j\}_{j \in [\ell]}$ of $V$ if, for all $i \in [m]$, there is some $j \in [\ell]$ with $V_i \subset W_j$.

\begin{lemma}[Regularity Lemma~\cite{KS96,Sze76}]\label{lem:regularity}
	For every $\eps>0$ and $m_0, \ell\geq 1$ there exists $M = M(\eps,m_0,\ell)$ such that the following holds. Let $G$ be a graph on $n  \geq M $ vertices whose edges are coloured in red and blue and let $d>0$.
	Let $\{W_i\}_{i \in [\ell]}$ be a partition of $V(G)$. Then there exists a partition $\{V_0, \dots, V_m\}$ of $V(G)$ and a subgraph $H$ of $G$ with vertex set $V(G) \sm V_0$ such that the following holds:
	\begin{enumerate}
		\item $m_0 \leq m \leq M$;
		\item $\{V_i\}_{i \in [m]}$ refines $\{W_i\cap V(H)\}_{i \in [\ell]}$;
		\item $|V_0| \leq \eps n$ and $|V_1| = \dots = |V_m| \leq \lceil\eps n \rceil$;
		\item $\deg_{H}(v) \geq \deg_G(v)-(d+\eps)n$ for each $v \in V(G) \sm V_0$;
		\item $H[V_i] $ has no edges for $i \in [m]$;
		\item \label{itm:reg-lemma-eps-regular} all pairs $(V_i,V_j)$ are $\eps$-regular and with density either 0 or at least $d$ in each colour in $H$.
	\end{enumerate}
\end{lemma}

Before we start with the proof, we will briefly describe the setup and proof strategy of Lemma~\ref{lem:path-forest}.
Consider a totally {$2$-coloured} complete graph $G=K_n$.
Denote the sets of red and blue vertices by $R$ and $B$, respectively.
For $\ell \geq 4,$ let $\{W_j\}_{j\in [\ell]}$ be a partition of $[n]$ such that each $W_j$ consists of at most $\lceil n/{\ell}\rceil$ subsequent vertices. 
The partition $\{W'_j\}_{j\in[2\ell]}$, with parts of the form $W_i\cap R$ and $W_i\cap B$, refines both $\{W_j\}_{j\in[\ell]}$ and $\{R,B\}$.
Suppose that $V_0 \cup \dots \cup V_m$ is a partition obtained from Lemma~\ref{lem:regularity} applied to $G$ and $\{W'_j\}_{j\in[2\ell]}$ with parameters $\eps$, $m_0$, $2\ell$ and $d$.
We define the $(\eps,d)$-\emph{reduced graph} $\RdG$ to be the graph with vertex set $V(\RdG) = [m]$ where $ij$ is an edge of $\RdG$ if and only if  if $(V_i,V_j)$ is an $\eps$-regular pair of density at least $d$ in the red subgraph of $H$ or in the blue subgraph of $H$.  Furthermore, we colour $ij$ red if $(V_i,V_j)$ is an $\eps$-regular pair of density at least $d$ in the red subgraph of $H$, otherwise we colour $ij$ blue. 
As $\{V_i\}_{i \in [m]}$ refines $\{R,B\}$, we can extend this to a total 2-colouring of $G'$ by colouring each vertex $i$ red, if $V_i \subset R$, and blue otherwise.
By relabelling the clusters, we can furthermore assume that $i<j$ if and only if $\max \{V_i\} < \max \{V_j\}$.
Note that, by choice of $\{W_j\}_{j\in [\ell]}$,  any two vertices in $V_i$ differ by at most $n/\ell$.
Moreover, a simple calculation (see \cite[Proposition 42]{KO}) shows that $\RdG$ has minimum degree at least $(1-d-3\eps)m$.

Given this setup, our strategy to prove Lemma~\ref{lem:path-forest} goes as follows.
First, we apply Lemma~\ref{lem:simple-forest} to obtain $t' \in [m]$ and a, red say, simple forest $F' \subset G'$ with $d(F',t') \approx \rioflat$.
Next, we turn $F'$ into a red path forest $F \subset G$.
For every isolated vertex $i \in V(F')$, this is straightforward as $V_i \subset R$ by the refinement property.
For every edge $ij \in E(F')$ with $i \in R$, we apply Lemma~\ref{fac:almost-spanning-path} to obtain a red path that almost spans $(V_i,V_j)$ and has both ends in $V_i$.
So the union $F'$ of these paths and vertices is indeed a red path forest.
Since the vertices in each $V_i$ do not differ too much, it will follow that $d(F,t) \approx \rioflat$ for $t = \max \{V_{t'}\}$.

\begin{proof}[Proof of Lemma~\ref{lem:path-forest}]
	Suppose we are given $\gamma>0$ and $k \in \mathbb{N}$ as input. Let $k_0,N \in \N$ and $\alpha >0$ be as in Lemma~\ref{lem:simple-forest} with input $\gamma/4$. We choose constants $d, \eps >0$ and $\ell,m_0 \in \N$ satisfying 
	\begin{equation*}
	\text{$2\sqrt{\eps}+\eps \leq 1/\ell,d \leq \alpha/8$ and $m_0 \geq 4N/d, 2k_0 N$.}
	\end{equation*}
	We obtain $M$ from  Lemma~\ref{lem:regularity} with input $\eps,m_0$ and $2\ell$.
	Finally, set $n_0 = 16 k  \ell  M  N $.
	
	Now let $n \geq n_0$ and suppose that $K_n$ is an ordered complete graph on vertex set $[n]$ and with a total $2$-colouring in red and blue.
	We have to show that there is an integer $t \in [k,n]$ and a monochromatic path forest $F$ such that $|V(F) \cap [t]| \geq (\rioflat - \gamma)t$.
	
	Denote the red and blue vertices by $R$ and $B$, respectively.
	Let $\{W'_j\}_{j \in [\ell]}$ refine $\{R,B\}$ as explained in the above setting.
	Let $\{V_0,\dots,V_m\}$ be a partition of $[n]$ with respect to $G=K_n$ and $\{W'_j\}_{j \in [\ell]}$ as detailed in Lemma~\ref{lem:regularity} with totally 2-coloured $(\eps,d)$-reduced graph $G''$ of minimum degree $\delta(G'') \geq (1-4d)m$.	
	Set $k'=\lfloor m/N \rfloor \geq k_0$	and observe that the subgraph $\RdG$ induced by $G''$ in $[k'N]$ satisfies $\delta(\RdG) \geq (1-8d)m \geq (1-\alpha)m$ as $m \geq 4N/d$.
	Thus we can apply Lemma~\ref{lem:simple-forest} with input $\RdG$, $k'$, $\gamma/4$ to obtain an integer $t' \in [k'/8,k'N]$ and a monochromatic (say red) simple forest $F' \subset \RdG$ such that $d(F',t') \geq \rioflat-\gamma/4$.
	
	Set $t = \max V_{t'} $.
	We have that $V_{t'} \subset W_j$ for some $j \in [\ell]$.
	Recall that $i<j$ if and only if $\max \{V_i\} < \max \{V_j\}$ for any $i,j \in [m]$.
	It follows that $V_i \subset [t]$ for all $i \leq t'$.
	Hence
	\begin{equation}\label{equ:t-geq-n/64N}
	t \geq t'|V_1| \geq \frac{k'}{8}|V_1| \geq \left\lfloor \frac{m}{N} \right\rfloor  \frac{(1-\eps)n}{8m} \geq \frac{n}{16N}. 
	\end{equation}
	This implies $t \geq k$ by choice of $n_0$. 
	Since $[t]$ is covered by $V_0 \cup W_j \cup \bigcup_{i \in [t']} V_i$, it follows that
	\begin{align}\label{equ:t-t'}
	t' |V_1| &\geq t -|V_0| - |W_j| \nonumber \\
	& \geq \left(1 - \eps\frac{n}{t} - \frac{4}{\ell}\frac{n}{t}\right)t\nonumber \\
	& \geq \left(1 -16\eps N - \frac{64N}{\ell}\right)t  ~~~~~\text{ (by }\eqref{equ:t-geq-n/64N}) \nonumber\\
	& \geq \left(1 - \frac{\gamma}{2} \right)t .
	\end{align}
	
	For every edge $ij \in E(F')$ with $V_i \subset R$, we apply Lemma~\ref{fac:almost-spanning-path} to choose a path $P_{ij}$ which starts and ends in $V_i$ and covers all but at most $2\sqrt{\eps}|V_{1}|$ vertices of each $V_i$ and $V_j$.
	We denote the isolated vertices of $F'$ by $I'$.
	For each $i \in I'$ we have $V_i \subset R$. Hence the  red path forest $F \coloneqq \bigcup_{i \in I'}V_i \cup \bigcup_{ij \in E(F')}P_{ij} \subset K_n$ satisfies
    
	\begin{align*}
	|V(F) \cap [t]|  &=  \sum_{i \in I'}|V_i\cap [t]| + \sum_{ij \in E(F')}|V(P_{ij})\cap [t]| \\
	&\geq   \sum_{i \in I' \cap [t']}|V_i| + \sum_{i \in V(F' - I') \cap [t'],}(|V_i|- 2\sqrt{\eps} |V_1|) \\
	&\geq (1-2\sqrt{\eps})|V_1| |V(F') \cap [t']| \\
	&\geq (1-2\sqrt{\eps}) \left(\frac{12+\sqrt{8}}{17}-\frac{\gamma}{4}\right) t' |V_1| \\
	&\overset{\eqref{equ:t-t'}}{\geq}  \left(\frac{12+\sqrt{8}}{17}-\gamma\right)t \nonumber
	\end{align*}
	as desired.
\end{proof}

\subsection{Upper density of simple forests}\label{sec:simple-forest}
In this section we prove Lemma~\ref{lem:simple-forest}. For a better overview, we shall define all necessary constants here.
Suppose we are given $\gamma'>0$ as input and set $\gamma = \gamma'/4$.
Fix a positive integer $N=N(\gamma)$ and let $0 < \alpha \leq \gamma/(8N)$.
The exact value of $N$ will be determined later on.
Let $k_0 = \lceil 8/\gamma \rceil$ and fix a positive integer $k \geq k_0$.
Consider a totally 2-coloured graph $G'$ on $n= kN$ vertices with minimum degree at least $(1-\alpha)n$.

Denote the sets of red and blue vertices by $R$ and $B$, respectively.
As it turns out, we will not need the edges inside $R$ and $B$.
So let $G$ be the spanning bipartite subgraph, obtained from $G'$ by deleting all edges within $R$ and $B$.
For each red vertex $v$, let $d_b(v)$ be the number of blue edges incident to $v$ in $G$.
Let $a_1\leq \dots\leq a_{|R|}$ denote the degree sequence taken by $d_b(v)$.
The whole proof of Lemma~\ref{lem:simple-forest} revolves around analysing this sequence.

Fix an integer $t=t(\gamma, N, k)$ and subset $R' \subset R$, $B' \subset B$.
The value of $t$ and nature of $R'$, $B'$ will be determined later.
The following two observations explain our interest in the sequence $a_1\leq \dots\leq a_{|R|}$.

\begin{claim}\label{obs:blue}
	If $a_j> j-t$ for all $1 \leq j \leq |R'|-1$, then there is a blue simple forest covering all but at most $t$ vertices of $R' \cup B$.
\end{claim}
\begin{proof}
	We write $R' = \{v_1,\dots,v_{|R'|}\}$ such that $d_b(v_i) \leq d_b(v_j)$ for every ${1 \leq i \leq j \leq |R'|}$.
	By assumption, we have $d_b(v_j) \geq a_j > j-t$ for all $1 \leq j \leq |R'|-1$.
	Thus we can greedily select a blue matching containing $\{v_{t}, v_{t+1}, \dots, v_{|R'|-1}\}$, which covers all but $t$ vertices of $R'$.
	Together with the rest of $B$, this forms the desired blue simple forest.
\end{proof}
\begin{claim}\label{obs:red}
	If $a_i<i+t$ for all $1 \leq i \leq |B'| -t$, then there is a red simple forest covering all but at most $t+\alpha n$ vertices of $R \cup B'$.
\end{claim}
\begin{proof}
	Let $X'$ be a minimum vertex cover of the red edges in the subgraph of $G$ induced by $R \cup B'$.
	If $|X'|\geq |B'|-t-\alpha n$, then by K\"onig's theorem there exists a red matching covering at least $|B'|-t-\alpha n$ vertices of $B'$.
	This together with the vertices in $R$ yields the desired red simple forest.

	Suppose now that $|X'| < |B'|-t-\alpha n$.  Since every edge between $R\setminus (X' \cap R)$ and $B'\setminus(X' \cap B')$ is blue, we have for every vertex $v$ in $R\setminus (X'\cap R)$, 
	\[d_b(v)\geq |B'|-|X'\cap B'|- \alpha n=|X'\cap R|+|B'|-|X'|-\alpha n >  |X'\cap R|+t.\]
	In particular, this implies $a_i\geq i+t$ for $i = |X'\cap R|+1$.
	So $|B'|-t+1 \leq |X'\cap R|+1$ by the assumption in the statement. Together with
	$$|X' \cap R| +1  \leq |X'| +1  < |B'| -t - \alpha n+1<|B'|-t+1,$$ we reach a contradiction.
	\end{proof}
Motivated by this, we introduce the following definitions.
\begin{definition}[Oscillation, $\ell^+(t)$, $\ell^-(t)$]
	Let $a_1, \dots, a_n$ be a non-decreasing sequence of non-negative real numbers. We define its \emph{oscillation} as the maximum value $T$, for which there exist indices $i,j \in [n]$ with $a_i-i\geq T$ and $j-a_j\geq T$. 
	For all $0<t\leq T$, set
	\begin{align*}
	\ell^+(t)&=\min\{i\in[n]\colon\ a_i\geq i+ t\}, \\ \ell^-(t)&=\min\{j\in[n]\colon\ a_j\leq j-t\}.
	\end{align*}
\end{definition}
\newcommand{\ml}{\ell}
Suppose that the degree sequence $a_1,\dots,a_{|R|}$ has oscillation $T$ and fix some positive integer $t \leq T$.  We define $\ell$ and $\lambda$ by
\begin{equation}\label{equ:definition-ml}
	\ml = \ell^+(t) + \ell^-(t)=\lambda t.
\end{equation}
The next claim combines Claims \ref{obs:blue} and \ref{obs:red} into a density bound of a monochromatic simple forest in terms of the ratio $\ml/t=\lambda$.
(Note that, in practice, the term $\alpha n$ will be of negligible size.) 
\begin{claim}\label{obs:stick-it-together}
	There is a monochromatic simple forest $F \subset G$ with $$d(F,\ml+t) \geq \frac{\ml-\alpha n}{\ml+t}=\frac{\lambda t-\alpha n}{(1+\lambda)t}.$$
\end{claim}
\begin{proof}
	Let $R' = R \cap [\ml + t]$ and $B' = B \cap [\ml + t]$ so that $\ell^+(t)+\ell^-(t)=\ell=|R'|+|B'|-t$.
	Thus we have either $\ell^-(t) \geq |R'|$ or $\ell^+(t) > |B'| - t$.
	If $\ell^-(t) \geq |R'|$, then $a_j > j-t$ for every $1 \leq j \leq |R'|-1$.
	Thus Claim~\ref{obs:blue} provides a blue simple forest $F$ covering all but at most $t$ vertices of $[\ell+t]$.
	On the other hand, if $\ell^+(t) > |B'| - t $, then $a_i < i+t$ for every $1 \leq i \leq |B'| -t$.
	In this case Claim~\ref{obs:red} yields a red simple forest $F$ covering all but at most $t+\alpha n$ vertices of $[\ell+t]$.
\end{proof}
Claim \ref{obs:stick-it-together} essentially reduces the problem of finding a dense linear forest to a problem about bounding the ratio $\ml/t$ in integer sequences.
It is, for instance, not hard to see that we always have $\ml \geq 2t$ (which, together with the methods of the previous two subsections, would imply the bound $\bar{d}(P)\geq 2/3$ of Erd\H{o}s and Galvin).  The following lemma provides an essentially optimal lower bound on $\ml/t=\lambda$.  Note that for $\lambda=4+\sqrt{8}$, we have $\frac{\lambda}{\lambda+1}=\rioflat$.
\begin{lemma}\label{lem:oscillation}
	For all $\gamma\in \mathbb{R}^+$, there exists $N\in \N$ such that, for all $k\in \mathbb{R}^+$ and all sequences with oscillation at least $kN$, there exists a real number $t\in[k,kN]$ with \[\ml := \ell^+(t)+\ell^-(t)\geq\left(4+\sqrt{8}-\gamma\right)t.\]
\end{lemma}
The proof of Lemma~\ref{lem:oscillation} is deferred to the last section.
We now finish the proof of Lemma~\ref{lem:simple-forest}.
Set $N=N(\gamma)$ to be the integer returned by Lemma~\ref{lem:oscillation} with input $\gamma = \gamma'/4$.
In order to use Lemma~\ref{lem:oscillation}, we have to bound the oscillation of $a_1,\dots,a_{|R|}$:
\begin{claim}\label{cla:oscillation}
	The degree sequence $a_1,\dots,a_{|R|}$ has oscillation $T \geq kN/8$ or there is a monochromatic simple forest $F \subset G$ with $d(F, n)\geq \rioflat- \gamma $.
\end{claim}
Before we prove Claim~\ref{cla:oscillation}, let us see how this implies Lemma~\ref{lem:simple-forest}.
\begin{proof}[Proof of Lemma~\ref{lem:simple-forest}]
	By Claim~\ref{cla:oscillation}, we may assume that the sequence $a_1, \dots, a_{|R|}$ has oscillation at least $kN/8$.
	By Lemma \ref{lem:oscillation}, there is a real number $t' \in \left[k/8,{kN}/8\right]$ with \[\ml = \ell^+(t')+\ell^-(t')\geq (4+\sqrt8-\gamma)t'.\] 
	Let $t=t(\gamma,N,k) =\left\lceil t'\right\rceil$. 
	Since the $a_i$'s are all integers, we have $\ell^+(t)=\ell^+(t')$ and $\ell^-(t)=\ell^-(t')$.
	Let $F \subset G$ be the monochromatic simple forest obtained from Claim~\ref{obs:stick-it-together}.
	As $n = kN$, $\ell \geq t'\geq  {k/8 \geq 1/\gamma}$, $\alpha \leq \gamma/(8N)$, and by~\eqref{equ:definition-ml}, it follows that
	\begin{align*}
	d(F,\ml+t) &\geq \frac{\ml-\alpha n}{\ml+t} = \frac{1-\alpha n/\ml}{1+\frac{t}{\ml}} \geq \frac{1-8\alpha N }{1+\frac{t'}{\ml}  {+ \frac{1}{\ml}}} \geq \frac{1 }{1+\frac{t'}{\ml} }  { - 2\gamma }
	\\ &\geq \frac{1}{1+\frac{1}{ 4+\sqrt8-\gamma}} - 2\gamma  = \frac{4+\sqrt8 - \gamma }{5+\sqrt8-\gamma} - 2\gamma 
	 \\&\geq \frac{4+\sqrt{8}}{5+\sqrt{8}}-4\gamma  =\frac{12+\sqrt{8}}{17}-\gamma',
	\end{align*}
	as desired.	
\end{proof}
To finish, it remains to show Claim~\ref{cla:oscillation}.
The proof uses K\"onig's theorem and is similar to the proof of Claim~\ref{obs:red}.
\begin{proof}[Proof of Claim~\ref{cla:oscillation}]
	Let $X$ be a minimum vertex cover of the red edges. If $|X|\geq {|B|}-(1/8+\alpha)n$, then K\"onig's theorem implies that there is a red matching covering all but at most $(1/8+\alpha)n$ blue vertices.
	Thus adding the red vertices, we obtain a red simple forest $F$ with $d(F, kN)\geq 7/8-\alpha\geq \rioflat- \gamma $.
	Therefore, we may assume that $|X|< {|B|}-(1/8+\alpha)n$. 
	Every edge between $R\setminus (X\cap R)$ and $B\setminus(X\cap B)$ is blue.
	So there are at least ${|R|}-|X\cap R|$ red vertices $v$ with 
	\[d_b(v) \geq {|B|}-|X\cap B|-\alpha n =|X\cap R|+{|B|}-|X|-\alpha n> |X\cap R|+n/8.\]
	This implies that $a_{i}\geq i+n/8$ for $i = |X\cap R|+1$. (See Figure \ref{fig:osc}.)
    
    \begin{figure}[ht]
	\centering
	\includegraphics[scale=1]{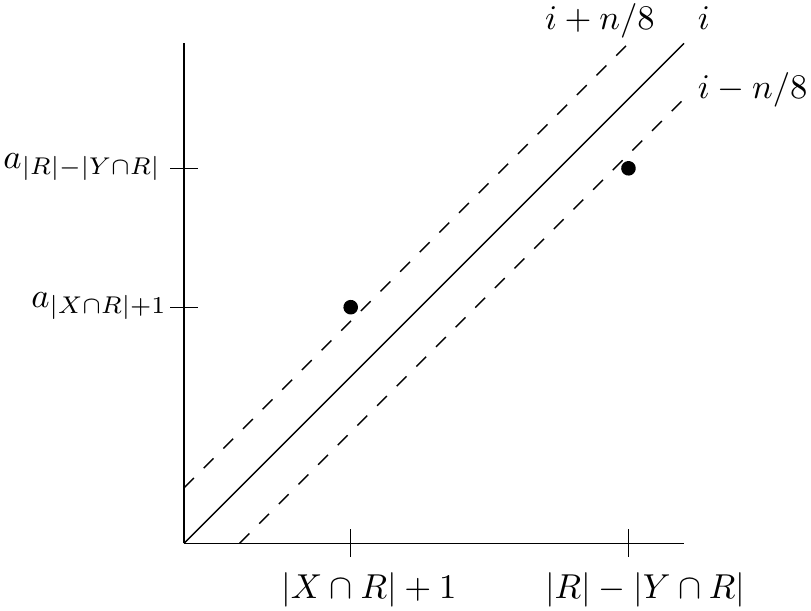}
	\caption{The sequence $a_1, \dots, a_{|R|}$ has oscillation at least $kN/8$.}\label{fig:osc}
\end{figure}
	
	Let $Y$ be a minimum vertex cover of the blue edges. Using K\"onig's theorem as above, we can assume that $|Y|\leq {|R|}-{n}/{8}$.
	Every edge between $R\setminus(Y\cap R)$ and $B\setminus (Y\cap B)$ is red.
	It follows that there are at least ${|R|}-|Y\cap R|$ red vertices $v$ with 
	\[d_b(v)\leq |Y\cap B|=|Y|-|Y\cap R|\leq {|R|}-|Y\cap R|-\frac{n}8.\]
	This implies that $a_{j}\leq j-{n}/8$ for $j = {|R|}-|Y\cap R|$.
	Thus $a_1, \dots, a_{|R|}$ has oscillation at least $n/8=kN/8$.
\end{proof}

\subsection{Sequences and oscillation}

We now present the quite technical proof of Lemma~\ref{lem:oscillation}.
We will use the following definition and related lemma in order to describe the oscillation from the diagonal.

\begin{definition}[$k$-good, $\uo(k)$, $\ue(k)$]
	Let $a_1,  \dots, a_n$ be a sequence of non-negative real numbers and let $k$ be a positive real number. We say that the sequence is $k$-\emph{good} if there exists an odd $i$ and an even $j$ such that $a_i\geq k$ and $a_j\geq k$.  If the sequence is $k$-good, we define for all $0<t\leq k$
	\begin{align*}
	\uo(t)&=a_1+\dots+a_{i_o-1}\quad \text{where }i_o=\min\{i\colon\ a_i\geq t, i\text{ odd}\}, \\
	\ue(t)&=a_1+\dots+a_{i_e-1}\quad \text{where }i_e=\min\{i\colon\ a_i\geq t, i\text{ even}\}.
	\end{align*}
\end{definition}

\begin{lemma}\label{lem:good-sequence}
	For all $\gamma\in \mathbb{R}^+$ there exists $N\in \N$ such that for all $k\in \mathbb{R}^+$ and all $(kN)$-{good} sequences, there exists a real number $t\in[k,kN]$ with\[\uo(t)+\ue(t)\geq \left(3+\sqrt{8}-\gamma\right)t.\]
\end{lemma}

First we use Lemma \ref{lem:good-sequence} to prove Lemma~\ref{lem:oscillation}.

\begin{proof}[Proof of Lemma~\ref{lem:oscillation}]
	Given $\gamma >0$, let $N$ be obtained from Lemma~\ref{lem:good-sequence}.
	Let $k \in \mathbb{R}^+$ and $a_1,\dots,a_n$ be a sequence with oscillation at least $kN$.
	Suppose first that $a_1\geq 1$.  Partition $[n]$ into a family of non-empty intervals
$I_1, \dots, I_r$ with the following properties: 
	\begin{itemize}
		\item For  every odd $i$ and every $j\in I_i$, we have $a_j\geq j$.
		\item For every even $i$ and every $j\in I_i$, we have $a_j< j$.
	\end{itemize} 
	Define $s_i=\max\left\{|a_j-j|\colon\ j\in I_i\right\}$. 
	Intuitively, this is saying that the values in the odd indexed intervals are ``above the diagonal'' and the values in the even indexed intervals are ``below the diagonal'' and $s_i$ is the largest gap between sequence values and the ``diagonal'' in each interval.  
	
	Since $a_1,\dots,a_n$ has oscillation at least $kN$, the sequence $s_1, \dots, s_r$ is $(kN)$-good and thus by Lemma \ref{lem:good-sequence}, there exists $t\in[k,kN]$ such that 
	\begin{equation}\label{eq:seq}
	\uo(t)+\ue(t)\geq \left(3+\sqrt{8}-\gamma\right)t.
	\end{equation} 
	Since the sequence $a_1, a_2, \dots, a_n$ is non-decreasing, $a_j-j$ can decrease by at most one in each step and thus we have $|I_i|\geq s_i$ for every $i\in [r-1]$. Moreover, we can find bounds on $\ell^+(t)$ and $\ell^-(t)$ in terms of the $s_i$:	
	\begin{itemize}
		\item $\ell^+(t)$ must lie in the interval $I_i$ with the smallest odd index $i_o$ such that $s_{i_o}\geq t$, therefore $\ell^+(t)\geq s_1+\dots +s_{i_o-1}=\uo(t)$.
		\item $\ell^-(t)$ must lie in the interval $I_j$ with the smallest even index $i_e$ such that $s_{i_e}\geq t$. Moreover, it must be at least the $t$-th element in this interval, therefore $\ell^-(t)\geq s_1+\dots +s_{i_e-1}+t=\ue(t)+t$.
	\end{itemize}
	Combining the previous two observations with \eqref{eq:seq} gives 
	\begin{equation*}
	\ell^+(t)+\ell^-(t)\geq \uo(t)+\ue(t)+t\geq \left(4+\sqrt{8}-\gamma\right)t, 
	\end{equation*}
	as desired. 
	
	If $0\leq a_1<1$, we start by partitioning $[n]$ into a family of non-empty intervals $I_1, \dots, I_r$ with the following properties: 
	\begin{itemize}
		\item For  every even $i$ and every $j\in I_i$, we have $a_j\geq j$.
		\item For every odd $i$ and every $j\in I_i$, we have $a_j< j$.
	\end{itemize} 
	From this point, the proof is analogous.
\end{proof}

Finally, it remains to prove Lemma~\ref{lem:good-sequence}. The proof is by contradiction and the main strategy is to find a subsequence with certain properties which force the sequence to become negative eventually.

\begin{proof}[Proof of Lemma~\ref{lem:good-sequence}]
	Let $\rho=3+\sqrt{8}-\gamma$ and let $m:=m(\rho)$ be a positive integer which will be specified later.  Suppose that the statement of the lemma is false for $N=6\cdot 4^m$ and let $a_1, \dots, a_n$ be an $(Nk)$-good sequence without $t$ as in the statement.  We first show that $a_i$ has a long strictly increasing subsequence.
	Set \[I=\{i\colon\ a_i\geq k, a_i>a_j\text{ for all }j<i\},\] denote the elements of $I$ by $i_1\leq i_2\leq \dots\leq i_r$ and let $a'_j=a_{i_j}$. Consider any $j \in [r-1]$ and suppose without loss of generality that $i_{j+1}$ is odd.
	For $\delta$ small enough, this implies $\uo(a'_j+\delta)=a_1+\dots+a_{i_{j+1}-1}\geq a'_1+\dots+a'_j$, and $\ue(a'_j+\delta)\geq a_1+\dots+a_{i_{j+1}} \geq a'_1+\dots+a'_{j+1}$.
	By assumption we have $\uo(a'_j+\delta)+\ue(a'_j+\delta)< \rho(a'_j+\delta)$.
	Hence, letting $\delta\rightarrow 0$ we obtain $2\left(a'_1+\dots+a'_j\right)+a'_{j+1}\leq\rho a'_j$, which rearranges to
	\begin{equation}\label{reca}
	a'_{j+1}\leq (\rho-2)a'_j-2\left(a'_1+\dots+a'_{j-1}\right).
	\end{equation} 
  In particular, this implies $a'_{j+1}\leq (\rho-2)a'_j <4a'_j$. Moreover, we have $a_1' \leq \uo(k)$ if $i_1$ is even and $a_1' \leq \ue(k)$ if $i_1$ is odd. Therefore,
	\[6k\cdot 4^m=kN\leq a_r'< 4^r\cdot a_1'\leq 4^r\max\{\uo(k), \ue(k)\}\leq 4^r(\uo(k)+\ue(k))<4^r\cdot \rho k<6k\cdot 4^r
	\]
	and thus $r\geq m$.  
	
	Finally, we show that any sequence of reals satisfying \eqref{reca}, will eventually become negative, but since $a_i'$ is non-negative this will be a contradiction.  
	
	We start by defining the sequence $b_1, b_2, \dots$ recursively by $b_1=1$ and $b_{i+1}=(\rho-2)b_i-2(b_1+\dots+b_{i-1})$.  Note that 
	\begin{align*}
	b_{i+1}&=(\rho-2)b_i-2(b_1+\dots+b_{i-1})\\
	&=(\rho-1)b_i-b_i-2(b_1+\dots+b_{i-1})\\
	&=(\rho-1)b_i-((\rho-2)b_{i-1}-2(b_1+\dots+b_{i-2}))-2(b_1+\dots+b_{i-1})\\
	&=(\rho-1)b_i-\rho b_{i-1}
	\end{align*}
	So equivalently the sequence is defined by, \[b_1=1,~b_2=\rho-2,\text{ and } b_{i+1}=(\rho-1)b_i-\rho b_{i-1}\text{ for }i\geq 2.\] 
	It is known that a second order linear recurrence relation whose characteristic polynomial has non-real roots will eventually become negative (see \cite{BW81}). Indeed, the characteristic polynomial $x^2-(\rho-1)x+\rho$ has discriminant $\rho^2-6\rho+1<0$ and so its roots $\alpha,\bar\alpha$ are non-real.  Hence the above recursively defined sequence has the closed form of $b_i=z\alpha^i+\bar z\bar\alpha^i=2\text{Re}\left(z\alpha^i\right)$ for some complex number $z$.  By expressing $z\alpha^i$ in polar form we can see that $b_m<0$ for some positive integer $m$.  Note that the calculation of $m$ only depends on $\rho$. 
	
	Now let $a'_1,  \dots, a'_m$ be a sequence of non-negative reals satisfying \eqref{reca}.  We will be done if we can show that $a_j'\leq a_1'b_j$ for all $1\leq j\leq m$; so suppose $a'_s>a'_1b_s$ for some $s$, and such that $\{a'_j\}_{j=1}^m$ and $\{a'_1b_j\}_{j=1}^m$ coincide on the longest initial subsequence. Let $p$ be the minimum value such that $a'_p\neq a'_1b_p$. Clearly $p>1$. Applying \eqref{reca} to $j=p-1$ we see that 
	\begin{align*}
	a'_p\leq (\rho-2)a'_{p-1}-2(a'_1+\dots+a'_{p-2})
	&=(\rho-2)a'_1b_{p-1}-2(a'_1b_1+\dots+a'_1b_{p-2})\\
	&=a_1'((\rho-2)b_{p-1}-2(b_1+\dots+b_{p-2}))=a'_1b_p
	\end{align*}
	and thus $a'_p<a'_1b_p$.
	
	Let $\beta={(a'_1b_p-a'_p)}/{a'_1}>0$. Now consider the sequence $a''_j$ where $a''_j=a'_j$ for $j<p$ and $a''_j=a'_j+\beta a'_{j-p+1}$ for $j\geq p$. Then $a''_p=a'_1b_p=a''_1b_p$. Clearly, this new sequence satisfies \eqref{reca} for every $ j < p$. Furthermore, we have
	\begin{align*}
	a''_{p+j}&=a'_{p+j}+\beta a'_{j+1}\\
	&\leq (\rho-2)a'_{p+j-1}-2\left(a'_1+\dots+a'_{p+j-2}\right) +\beta(\rho-2)a'_j-2\beta\left(a'_1+\dots+a'_{j-1}\right)\\
	&=(\rho-2)a''_{p+j-1}-2\left(a''_1+\dots+a''_{p+j-2}\right)
	\end{align*}
    for every $j \geq 0$. Hence, the whole sequence satisfies \eqref{reca}.
	We also have $a''_s\geq a'_s>a'_1b_s=a''_1b_s$.  This contradicts the fact that $a_j'$ was such a sequence which coincided with $a_1'b_j$ on the longest initial subsequence.
\end{proof}


\section*{Acknowledgments} 
This project began as part of the problem session of the ``Extremal Graph Theory and Ramsey Theory'' focus week of the ``Rio Workshop on Extremal and Structural Combinatorics'' held at IMPA, Rio de Janeiro, Brazil in January 2018.  We thank the organisers of the workshop and IMPA for the stimulating working environment.

We also thank the referees for their careful reading of the paper and their helpful suggestions.  

\bibliographystyle{amsplain}


\begin{aicauthors}
\begin{authorinfo}[jc]
  Jan Corsten\\ 
  London School of Economics, Department of Mathematics, London WC2A 2AE.\\
  j\imagedot{}corsten\imageat{}lse\imagedot{}ac\imagedot{}uk
\end{authorinfo}
\begin{authorinfo}[ld]
  Louis DeBiasio\\ 
  Miami University, Department of Mathematics, Oxford, OH, 45056, United States.\\
  debiasld\imageat{}miamioh\imagedot{}edu
\end{authorinfo}
\begin{authorinfo}[al]
  Ander Lamaison\\ 
  Institut f\"ur Mathematik, Freie Universit\"at Berlin and Berlin Mathematical School, Berlin, Germany.\\
  lamaison\imageat{}zedat\imagedot{}fu-berlin\imagedot{}de
\end{authorinfo}
\begin{authorinfo}[rl]
  Richard Lang\\ 
  University of Waterloo, Combinatorics \& Optimization, Waterloo, ON, N2L 3G1, Canada.\\
  r7lang\imageat{}uwaterloo\imagedot{}cl
\end{authorinfo}
\end{aicauthors}

\end{document}